\documentclass{article}
\usepackage{amsfonts,amssymb,amscd, amsmath}
\usepackage{latexsym, graphicx}
\usepackage[affil-it]{authblk}
\usepackage{algorithmic}
\usepackage{algorithm}
\usepackage{tikz}

\usepackage{enumerate}

\def\qed{\hfill $\Box$\vspace{2ex}}

\newtheorem{theorem} {Theorem}[section]
\newtheorem{lemma}[theorem]{Lemma}
\newtheorem{observation}[theorem]{Observation}
\newtheorem{corollary}[theorem]{Corollary}
\newtheorem{problem}[theorem]{Problem}

\usepackage{xcolor}

\newenvironment{proof}{\noindent {\it Proof:~}}{\hfill $\Box$\smallskip\par}

\def\inst#1{$^{#1}$}
\begin{document}

\title{On the structure of (banner, odd hole)-free graphs\thanks{Research
support by Natural Sciences and Engineering Research Council of
Canada.}}
\author{Ch\'inh T. Ho\`ang\inst{1}
           }
\date{}
\maketitle
\begin{center}
{\footnotesize

\inst{1} Department of Physics and Computer Science, Wilfrid
Laurier University,
Waterloo, Ontario,  Canada, N2L 3C5\\
\texttt{choang@wlu.ca} }

\end{center}

\begin{abstract}
A hole is a chordless  cycle with at least four vertices. A hole
is odd if it has an odd number of vertices. A banner is a graph
which consists of a hole on four vertices and a single vertex with
precisely one neighbor on the hole. We prove that a (banner, odd
hole)-free graph is perfect, or does not contain a stable set on
three vertices, or contains a homogeneous set. Using this
structure result, we design a polynomial-time algorithm for
recognizing (banner, odd hole)-free graphs. We also design
polynomial-time algorithms to find, for such a graph, a minimum
coloring and largest stable set. A graph $G$ is perfectly
divisible if every induced subgraph $H$ of $G$ contains a set $X$
of vertices such that $X$ meets all largest cliques of $H$, and
$X$ induces a perfect graph. The chromatic number of a perfectly
divisible graph $G$ is bounded by $\omega^2$ where $\omega$
denotes the number of vertices in a largest clique of $G$. We
prove that (banner, odd hole)-free graphs are perfectly divisible.
\end{abstract}

\section{Introduction}\label{sec:introduction}
A {\em hole} is a chordless cycle with at least four vertices.  An
{\it antihole} is the complement of a hole. A hole is {\it odd} if
it has an odd number of vertices. A graph is odd-hole-free if it
does not contain, as an induced subgraph, an odd hole.
Odd-hole-free graphs are studied in connection with perfect graphs
(definitions not given here will be given later.) It follows from
a result of  Kr\'al,  Kratochv\'il, Tuza and  Woeginger
(\cite{KraKra2001}) that it is NP-hard to color an odd-hole-free
graph. In contrast, there is a polynomial-time algorithm
(Gr\"otschel, Lov\'asz and Schrijver \cite{GroLov1984}) to find a
minimum coloring of a graph with no odd holes and no odd
antiholes.

Chudnovsky,  Cornu\'ejols,  Liu, Seymour, and  Vu\v{s}kovi\'c
(\cite{ChuCor2005}) designed a polynomial-time algorithm for
finding an odd hole or odd antihole, if one exists, in a graph.
However, the complexity of recognizing odd-hole-free graphs is
unknown. Conforti, Cornu\'ejols,  Kapoor, and Vu\v{s}kovi\'c
(\cite{ConCor2004}) found a decomposition of odd-hole-free graphs
by ``double star-cutsets'' and ``2-joins'' into certain ``basic
graphs'', but this decomposition does not seem to help in
designing a polynomial-time recognition algorithm for
odd-hole-free graphs.

A {\it banner} is the graph which consists of a hole on four
vertices and a single vertex with precisely one neighbor on the
hole. Banner-free graphs generalize the well studied class of
claw-free graphs. In this paper, we study (banner, odd hole)-free
graphs. We will show that if a (banner, odd hole)-free graph is
not perfect, then either it contains a homogeneous set or its
stability number is at most two. We will use this structural
result to design polynomial-time algorithms for (i) recognizing a
(banner, odd hole)-free graph, (ii) finding an optimal coloring of
a (banner, odd hole)-free graph, and (iii) finding a largest
stable set of a (banner, odd hole)-free graphs. In contrast,
results in the literature show that, for a (banner, odd hole)-free
graph, finding a largest clique and finding a minimum cover by
cliques are both NP-hard problems. We will discuss these facts in
Section~\ref{sec:optimization}. We note the algorithm of Gerber,
Hertz, and  Lozin (\cite{GerHer2004}) that finds in polynomial
time the largest stable set of a (banner, $P_8$)-free graph. Also,
the result of Kr\'al et al. (\cite{KraKra2001}) shows that it is
NP-hard to find a minimum coloring of an odd-hole-free graph.

The notion of ``robust'' algorithms was introduced by Raghavan
and Spinrad (\cite{RagSpi2003}).  A {\it robust} algorithm for
problem $P$ on domain $C$ must solve $P$ correctly for every input
in $C$. For input not in $C$, the algorithm may produce correct
output for problem $P$, or answer that the input is not in $C$.
Our optimization algorithms are robust.

A graph $G$ with at least one edge  is {\it $k$-divisible}  if the
vertex-set of each of its induced subgraphs $H$ with at least one
edge can be partitioned into $k$ sets, none of which contains a
largest clique of $H$. It is easy to see that the chromatic number
of a $k$-divisible graph is at most $k^{\omega -1}$. It was conjectured
by Ho\`ang and McDiarmid (\cite{HoaMcd2002}, \cite{HoaMcd1999})
that every odd-hole-free graph is 2-divisible. They proved the
conjecture for claw-free graphs (\cite{HoaMcd2002}). We will prove
the conjecture for banner-free graphs.

A graph $G$ is {\it perfectly divisible} if every induced subgraph
$H$ of $G$ contains a set $X$ of vertices such that $X$ meets all
largest cliques of $H$, and $X$ induces a perfect graph. The
chromatic number of a perfectly divisible graph $G$ is bounded by
$\omega^2$ where $\omega$ denotes the number of vertices in a
largest clique of $G$. We will prove that (banner, odd hole)-free
graphs are perfectly divisible.

In Section~\ref{sec:background}, we give the definitions used in
this paper and discuss background results that are used by our
algorithms. In Section~\ref{sec:structure},  we prove our theorems
on the structure of (banner, odd hole)-free graphs.
In Section~\ref{sec:recognition}, we give two polynomial-time
algorithms for recognizing a (banner, odd hole)-free graph.  In
Section~\ref{sec:optimization}, we give polynomial-time algorithms
for finding a minimum coloring and a largest stable set of a
(banner, odd hole)-free graph.
In Section~\ref{sec:divisibility}, we prove that (banner, odd
hole)-free graphs are 2-divisible. Finally, in
Section~\ref{sec:perfect-divisibility}, we prove that (banner, odd
hole)-free graphs are perfectly divisible.

\section{Definitions and background}\label{sec:background}
\subsection{Definitions}
The {\it claw} and {\it banner} are represented in Figure~\ref{fig:claw-banner}.
Let $K_t$ denote the clique on $t$ vertices. Let $C_k$
(respectively, $P_k$) denote the cordless cycle (respectively,
path) on $k$ vertices. The {\it girth} of a graph is the length of
its shortest cycle. A {\it hole} is the graph $C_k$ with $k \geq
4$. A hole is {\it odd} if it has an odd number of vertices.  An
{\it antihole} is the complement of a hole. A clique on three
vertices is called a {\it triangle}. For a given graph $H$, it is
customary to let {\it co}-$H$ denote the complement of $H$. Thus,
a co-triangle is the complement of the triangle. Let $L$ be a collection 
of graphs. A graph $G$ is $L$-free if $G$ does not contain
an induced subgraph isomorphic to a graph in $L$. In particular,
a graph is {\it (banner, odd hole)-free} if it does not contain an
induced subgraph isomorphic to a banner or an odd hole.
Let $G_1, G_2$ be two vertex-disjoint graphs. The union of $G_1$
and $G_2$ is denoted by $G_1 + G_2$. If $k$ is an integer, then $k
G_1$ denotes the union of $k$ disjoint copies of $G_1$. For a
graph $G$ and a set $X \subseteq V(G)$, $G[X]$ denotes the
subgraph of $G$ induced by $X$.

\begin{figure}
\newcommand{\ang}{17}
\newcommand{\sep}{15}
\begin{tikzpicture}[scale=2]
\tikzstyle{every line}=[thick];
\tikzstyle{every node}=[draw,shape=circle,fill = black, minimum size=0.05cm];

 \draw (2,0) node[label=below: (a) claw](firstNode){}
          -- (2,0.5) node (middleNode){};
 \draw (middleNode) -- (2,1) node{};
  \draw (middleNode) -- (2.5,0.5) node{};

  \draw (4,0) node[label=below: (b) banner](lower){}
          -- (4,0.5) node (middle){};
 \draw (middle) -- (4,1) node(upper){};
  \draw (middle) -- (4.7,0.5) node(middleRight){};
  \draw (upper) -- (4.7,1) node(topRight){};
  \draw (middleRight) -- (topRight);

 \end{tikzpicture}
\caption{The claw and the banner}\label{fig:claw-banner}

\end{figure}
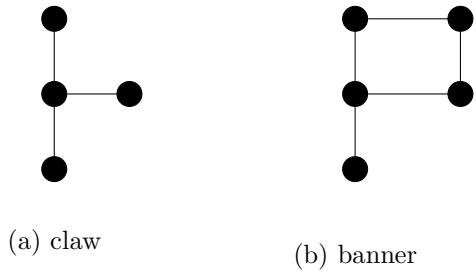

Let $G$ be a graph. Then $\chi(G)$ denotes the chromatic number of
$G$, and $\omega(G)$ denotes, the {\it clique number} of $G$, that
is, the number of vertices in a largest clique of $G$. Let $v$ be a vertex and $X$ be a set of vertices of $V(G)$. We say that $v$ is {\it X-complete} if $v$ is adjacent to all vertices 
of $X$, and $v$ is {\it X-anticomplete} if $v$ is non-adjacent to all vertices of $X$.
A set $H \subset V(G)$
of a graph $G$ is {\it homogeneous} if $2 \leq |H| < |V(G)| $ and
every vertex in $G - H$ is either $H$-complete or $H$-anticomplete. A graph is {\it prime} if it
contains no homogeneous set. In algorithm analysis, it is
customary to let $n$, respectively, $m$, denote the number of
vertices, respectively, edges, of the input graph.

A graph $G$ is {\it perfect} if $\chi(H) = \omega(H)$ for every
induced subgraph $H$ of $G$. A graph is {\it Berge} if it does not
contain as induced subgraph an odd hole or odd antihole. Two important results are known about
perfect graphs.  The Perfect Graph Theorem, proved by Lov\'asz
(\cite{Lov1972}), states that a graph is perfect if and only if
its complement is. The Strong Perfect Graph Theorem, proved by
Chudnovsky,  Robertson, Seymour, and  Thomas (\cite{ChuRob2006}),
states that a graph is perfect if and only if it is Berge. Both of the above results were long
standing open problems proposed by Berge (\cite{Ber1961}).
Chudnovsky et al. (\cite{ChuCor2005}) designed  an $O(n^9)$
algorithm for recognizing Berge graphs. Gr\"otschel et al.
(\cite{GroLov1984}) designed a polynomial-time and robust
algorithm for finding a largest clique and a minimum coloring of a
perfect graph. Corollary 1 in Kr\'al et al. (\cite{KraKra2001})
shows that it is NP-hard to compute the chromatic number of a
($2K_2, K_2 + 2K_1, 4K_1, C_5$)-free graph. It was pointed out to
the author by K. Cameron that this result implies that it is
NP-hard to color an odd-hole-free graph.

\subsection{Modular Decomposition}\label{sec:modular-decomposition}
Our recognition algorithms use the  well-studied modular
decomposition which we now discuss. Let $G$ be a graph. A {\it
module} is a non-empty set $M$ of vertices such that every vertex in $G-M$
is either $M$-complete or $M$-anticomplete.
Trivially, $\{x\}$ for any $x \in V(G)$, and $V(G)$ are modules. A
module $M$ is {\it non-trivial} if $2 \leq |M| < |V(G)|$, that is,
a non-trivial module is a homogeneous set. (Here, we are bound to
use two different names for essentially the same structure. In the
study of graph coloring, ``homogeneous set'' is often used,
whereas in graph algorithm design, ``module'' is used.)

Two sets $X$ and $Y$ {\it overlap} if $X-Y,Y-X$, and $Y \cap X$ are all non-empty.
$X \Delta Y$ denotes the symmetric difference of $X$ and $Y$, that is, $(X - Y)
\cup (Y - X)$. It is easy to see that if $X$ and $Y$ are overlapping modules then
the following sets are also modules: $X-Y, Y - X, X \cup Y, X \cap Y$, and $X \Delta Y$.
A module is {\it strong} if it does not overlap another module.  A non-trivial module $M$
of  a graph $Q$ is {\it maximal} if
there does not exist module $M'$ of $Q$ such that $M \subset M'
\subset Q$. If both $G$ and $\overline{G}$ are connected, then the maximal modules are strong.

Modular decomposition refers to the process of partitioning the vertices
of a graph $G$ into its strong modules ${\cal P} = (M_1, M_2, \ldots, M_t)$. 
Each $M_i$, if it is not prime,
is recursively decomposed.
The procedure stops when every module has a single vertex.
To a modular partition ${\cal P}$, we associate a {\it quotient} graph $G_{ {\cal P} }$ whose
vertices are the modules defined in ${\cal P}$;
two vertices $v_i$ and $v_j$ of $G_{ {\cal P} }$ are adjacent if and only if
the corresponding modules $M_i$ and $M_j$ are adjacent in $G$.
For the graph in Figure~\ref{fig:decomp}, the modular decomposition is
${\cal P} = \{ \{1\}, \{2,3\}, \{4,5\}, \{6,7,8,9,10\} \}$. The quotient graph $G_{ {\cal P} }$ is the $P_4$.
For each module $M$ in
${\cal P}$, if $M$ is not prime, then $M$ is recursively decomposed into strong modules.
For example,
the module $\{6,7,8,9,10\}$ is decomposed into $\{ \{6\}, \{7\}, \{8\}, \{ 9, 10 \} \}$.
The result of the modular decomposition can be represented as
a tree.

A module which induces a disconnected subgraph in the graph is a
{\em parallel module}. A module which induces a disconnected
subgraph in the complement of the graph is a {\em series module}.
A module which induces a connected subgraph in the graph as well
as in the complement of the graph is a {\em neighborhood module}.

If the current set $Q$ of vertices induces a disconnected
subgraph, $Q$ is decomposed into its components. A node labeled
$P$ (for parallel) is introduced, each component of $Q$ is
decomposed recursively, and the roots of the resulting subtrees
are made children of the $P$ node. If the complement of the
subgraph induced by current set $Q$ is disconnected, $Q$ is
decomposed into the components of the complement. A node labeled
$S$ (for series) is introduced, each component of the complement
of $Q$ is decomposed recursively, and the roots of the resulting
subtrees are made children of the $S$ node. Finally, if the
subgraph induced by the current set $Q$ of vertices and its
complement are connected, then $Q$ is decomposed into its maximal modules; it is known (Gallai \cite{Gal1967}) that in
this case, each vertex of $Q$ belongs to a unique maximal 
module of $Q$. A node labeled $N$ (for neighborhood) is
introduced, each maximal  module of $Q$ is decomposed
recursively, and the roots of the resulting subtrees are made
children of the $N$ node. A graph and its modular decomposition
tree are shown in Figure~\ref{fig:decomp}.
\setlength{\unitlength}{1.2pt}
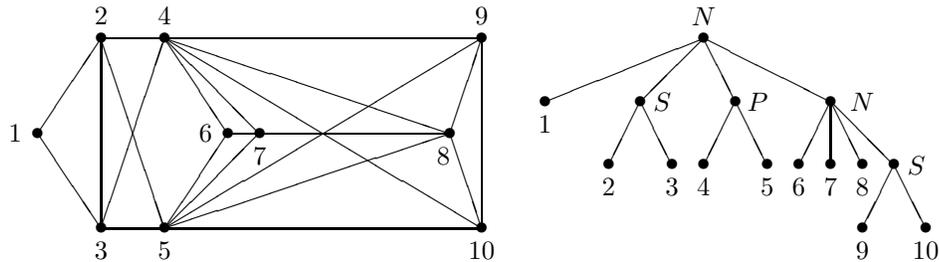
\begin{figure}[ht]
\center \mbox{
\begin{picture}(300, 90)
%
%
\put(10,40){\makebox(0,0){$\bullet$}}
\put(3,40) {\makebox(0,0){$1$}}
%
\put(30,70){\makebox(0,0){$\bullet$}}
\put(30,77) {\makebox(0,0){$2$}}
%
\put(30,10){\makebox(0,0){$\bullet$}}
\put(30,3) {\makebox(0,0){$3$}}
%
\put(50,70){\makebox(0,0){$\bullet$}}
\put(50,77) {\makebox(0,0){$4$}}
%
\put(50,10){\makebox(0,0){$\bullet$}}
\put(50,3) {\makebox(0,0){$5$}}
%
\put(70,40){\makebox(0,0){$\bullet$}}
\put(63,40) {\makebox(0,0){$6$}}
%
\put(80,40){\makebox(0,0){$\bullet$}}
\put(80,33) {\makebox(0,0){$7$}}
%
\put(140,40){\makebox(0,0){$\bullet$}}
\put(138,33) {\makebox(0,0){$8$}}
%
\put(150,70){\makebox(0,0){$\bullet$}}
\put(150,77) {\makebox(0,0){$9$}}
%
\put(150,10){\makebox(0,0){$\bullet$}}
\put(150,3) {\makebox(0,0){$10$}}
%
 \put(10,40) {\line(2,3){20}} 
 \put(10,40) {\line(2,-3){20}} 
 \put(30,70) {\line(0,-1){60}} 
 \put(30,70) {\line(1,0){20}} 
 \put(30,70) {\line(1,-3){20}} 
 \put(30,10) {\line(1,3){20}} 
 \put(30,10) {\line(1,0){20}} 
 \put(50,70) {\line(2,-3){20}} 
 \put(50,70) {\line(1,-1){30}} 
 \put(50,70) {\line(3,-1){90}} 
 \put(50,70) {\line(1,0){100}} 
 \put(50,70) {\line(5,-3){100}} 
 \put(50,10) {\line(2,3){20}} 
 \put(50,10) {\line(1,1){30}} 
 \put(50,10) {\line(3,1){90}} 
 \put(50,10) {\line(5,3){100}} 
 \put(50,10) {\line(1,0){100}} 
 \put(70,40) {\line(1,0){10}} 
 \put(80,40) {\line(1,0){60}} 
 \put(140,40) {\line(1,3){10}} 
 \put(140,40) {\line(1,-3){10}} 
 \put(150,70) {\line(0,-1){60}} 
%
%
%
\put(220,70){\makebox(0,0){$\bullet$}}
\put(220,77) {\makebox(0,0){$N$}}
%
\put(170,50){\makebox(0,0){$\bullet$}}
\put(170,43) {\makebox(0,0){$1$}}
%
\put(200,50){\makebox(0,0){$\bullet$}}
\put(207,50) {\makebox(0,0){$S$}}
%
\put(230,50){\makebox(0,0){$\bullet$}}
\put(237,50) {\makebox(0,0){$P$}}
%
\put(260,50){\makebox(0,0){$\bullet$}}
\put(270,50) {\makebox(0,0){$N$}}
%
\put(190,30){\makebox(0,0){$\bullet$}}
\put(190,23) {\makebox(0,0){$2$}}
%
\put(210,30){\makebox(0,0){$\bullet$}}
\put(210,23) {\makebox(0,0){$3$}}
%
\put(220,30){\makebox(0,0){$\bullet$}}
\put(220,23) {\makebox(0,0){$4$}}
%
\put(240,30){\makebox(0,0){$\bullet$}}
\put(240,23) {\makebox(0,0){$5$}}
%
\put(250,30){\makebox(0,0){$\bullet$}}
\put(250,23) {\makebox(0,0){$6$}}
%
\put(260,30){\makebox(0,0){$\bullet$}}
\put(260,23) {\makebox(0,0){$7$}}
%
\put(270,30){\makebox(0,0){$\bullet$}}
\put(270,23) {\makebox(0,0){$8$}}
%
\put(280,30){\makebox(0,0){$\bullet$}}
\put(287,30) {\makebox(0,0){$S$}}
%
\put(270,10){\makebox(0,0){$\bullet$}}
\put(270,3) {\makebox(0,0){$9$}}
%
\put(290,10){\makebox(0,0){$\bullet$}}
\put(290,3) {\makebox(0,0){$10$}}
%
 \put(220,70) {\line(-5,-2){50}} 
 \put(220,70) {\line(-1,-1){20}} 
 \put(220,70) {\line(1,-2){10}} 
 \put(220,70) {\line(2,-1){40}} 
 \put(200,50) {\line(-1,-2){10}} 
 \put(200,50) {\line(1,-2){10}} 
 \put(230,50) {\line(-1,-2){10}} 
 \put(230,50) {\line(1,-2){10}} 
 \put(260,50) {\line(-1,-2){10}} 
 \put(260,50) {\line(0,-1){20}} 
 \put(260,50) {\line(1,-2){10}} 
 \put(260,50) {\line(1,-1){20}} 
 \put(280,30) {\line(-1,-2){10}} 
 \put(280,30) {\line(1,-2){10}} 
\end{picture}
}   
\caption{A graph and its modular decomposition tree} \label{fig:decomp}
\end{figure}
\begin{theorem} {\rm (McConnell and Spinrad \cite{MccSpi1999})} \label{thm:decomp}
The modular decomposition tree of a graph is unique and it can be
constructed in O($m+n$) time. \qed
\end{theorem}

To analyze our algorithms, it will be more convenient to consider
the modular decomposition as a ``binary'' decomposition. Let $H$
be a homogeneous set of a graph $G$. The decomposition produces
two graphs: $H$ and $G[ (V(G)-H) \cup \{h\}]$ for a vertex $h \in
H$. Note that $G[(V(G) - H) \cup \{h_1\}]$ is isomorphic to
$G[(V(G) - H) \cup \{h_2\}]$ for any two vertices $h_1, h_2$ in
$H$. If either of the graphs $H$ and $G[ (V(G)-H) \cup \{h\}]$ is
not prime, then it is recursively decomposed. This process is
summarized by Algorithm~\ref{al:modular-decomposition}.
\begin{algorithm}
\caption{HOMOGENEOUS-SET-DECOMPOSITION}\label{al:modular-decomposition}
\begin{algorithmic}
 \STATE{\textbf{input:}} graph $G$
 \STATE{\textbf{output:}} a set $L$ containing prime induced subgraphs of $G$ produced by modular decomposition.
 \STATE{\textbf{initialization:}} $L \leftarrow \emptyset$, $W  \leftarrow \{G\}$ ($W$ contains the graphs to be decomposed)
 \STATE
 \IF{ $W = \emptyset$}
    \STATE {\bf stop}
 \ENDIF

  \REPEAT
      \STATE take one graph $F$ from the set $W$ and remove $F$ from $W$
       \IF { $F$ is prime}
          \STATE add $F$ to $L$
     \ELSE
         \STATE find a homogeneous set $H$ of $F$
         \STATE  add to $W$ the two graphs: $H$ and $  F[ (V(F) - H \cup \{h\} \}$ for an arbitrary vertex $h \in H$
    \ENDIF
   \UNTIL{ $W$ is empty}
\end{algorithmic}
\end{algorithm}
Note that the homogeneous sets considered by the algorithm  are
not necessarily maximal. We will reproduce an argument used by
Ho\`ang and Reed (\cite{HoaRee1989}) to show that the number of
prime graphs produced by Algorithm~\ref{al:modular-decomposition}
is at most $n -1$ for an input graph with at least two vertices.
Let $p(G)$ be the number of prime graphs produced by
Algorithm~\ref{al:modular-decomposition} on a graph $G$ with $|G|$
vertices. We are going to prove, for a graph $G$ with at least two
vertices,  $p(G) \leq  |G| -1$ by induction on the number of
vertices of $G$. If $G$ is prime, then $p(G) = 1$. Let $H$ be a
homogeneous set of $G$ that is used by
Algorithm~\ref{al:modular-decomposition} in the decomposition
step. The algorithm will decompose $H$ and and $ L=  G[ (V(G) - H
\cup \{h\} \}$ for a vertex $h$ of $H$. By the induction
hypothesis, we have $p(G) = p(H) + p(F)  \leq  (|H| - 1 ) + ( |F|
-1 )  = (|H| +|F|) - 2$. Since $|H| + |F| = |G| + 1$, we have
$p(G) \leq |G| -1$.

We note there are more general methods to handle modular decompositions of graphs. In particular,  Cunningham \cite{Cun1982} studied a decomposition of directed graphs, and  Rao \cite{Rao2008} studied the split decomposition. Their methods are more general but also more complicated than the method on modular decomposition presented here. 
\section{The structure of (banner, odd hole)-free graphs}\label{sec:structure}
Chv\'atal and Sbihi (\cite{ChvSbi1998}) designed a polynomial-time
algorithm for recognizing claw-free perfect graphs. In the
process, they found a proof of the ``Ben Rebea's Lemma'' below.
\begin{lemma}\label{lem:benrebea}[Ben Rebea's Lemma \cite{ChvSbi1998} ]
Let $G$ be a connected claw-free graph with $\alpha(G) \geq 3$. If
$G$ contains an odd antihole, then $G$ contains a $C_5$.
\end{lemma}
It is this Lemma that inspires our two
Lemmas~\ref{lem:triangle-free-component} and~\ref{lem:co-triangle}
and Theorem~\ref{thm:main} below.
\begin{lemma}\label{lem:triangle-free-component}
Let $G$ be a (banner, $C_5$)-free graph containing an odd antihole $A$ such that no co-triangle of $G$ contains two vertices of $A$. Let $O$ be the component of  $\overline{G}$ that contains $A$. Then, 
\begin{description}
\item [(i)] $G[V(O)]$ is co-triangle-free, and 
\item [(ii)] if $\alpha(G) \geq 3$, then $V(O)$ is a homogeneous set of $G$.
\end{description}
\end{lemma}
\begin{proof}
Enumerate the vertices of $A$ as $v_1, v_2 , \ldots , v_k$ such that $v_i v_{i+1}$ is a non-edge
of $G$ with the subscripts taken module $k$, with $k \geq 7$. We will first establish the following claim.
%
\begin{subequations}\label{first:main}
\begin{equation}\tag{\ref{first:main}}
\begin{minipage}{0.9\linewidth}
No vertex in $A$ belongs to a co-triangle. 
\end{minipage}\end{equation}
Suppose some vertex $v_i$ of $A$ forms a co-triangle with some two vertices
$u_1, u_2$ in $G-A$. By the hypothesis of the Lemma, we may assume $v_{i+1}$ is adjacent to both $u_1, u_2$. Similarly,
\begin{equation}
\begin{minipage}{0.9\linewidth} \label{first:1}
$v_{i-1}$ is adjacent to both $u_1, u_2$.
\end{minipage}\end{equation}
We will show
\begin{equation}
 \begin{minipage}{0.9\linewidth}
\label{eq:non-adjacent-to-one}
Each of the vertices in $\{ v_{i-2}, v_{i+2} \}$ is non-adjacent to at least one vertex in $\{ u_1, u_2\}$.
\end{minipage}\end{equation}
If $v_{i+2}$ is adjacent to both $u_1, u_2$, then the graph
$G[\{v_i, v_{i+1}, v_{i+2}, u_1, u_2\}]$ induces a banner. So, $v_{i+2}$ is non-adjacent to at least one vertex
in the set $\{u_1, u_2\}$. By symmetry, (\ref{eq:non-adjacent-to-one}) holds. Next, we will show
\begin{equation}
\begin{minipage} {0.9\linewidth} \label{eq:non-adjacent-to-all}
Each of the vertices in $\{ v_{i-2}, v_{i+2} \}$ is non-adjacent to all vertices in $\{ u_1, u_2\}$.
\end{minipage}\end{equation}
Suppose $v_{i+2}$ is adjacent to $u_2$ and non-adjacent to $u_1$. We have the
following implications.
\begin{description}
  \item  $u_2 v_{i+3} \in E(G)$, for otherwise $G[\{v_i, v_{i+1}, v_{i+2}, v_{i+3}, u_2\}]$
induces a $C_5$,
  \item   $u_1 v_{i+3} \not\in E(G)$, for otherwise $G[\{v_{i}, v_{i+2}, v_{i+3}, u_2, u_1\}]$
induces a banner,
  \item  $u_1 v_{i-2} \in E(G)$, for otherwise $G[\{v_{i-2}, v_{i-1}, v_{i+2}, v_{i+3}, u_1\}]$
induces a banner,
  \item  $u_2 v_{i-2} \not\in  E(G)$ by (\ref{eq:non-adjacent-to-one}).
\end{description}

Now, $G[\{v_{i-2}, v_{i+2}, v_{i+3}, u_1, u_2\}]$ induces a banner.
We have established (\ref{eq:non-adjacent-to-all}).

It follows from (\ref{eq:non-adjacent-to-all}) that, for all $j$,  $v_{i+2j}$ is non-adjacent to
both $u_1, u_2$ with the subscripts taken modulo $k$.
Since $A$ is odd, this implies $v_{i-1}$ is non-adjacent to both
$u_1, u_2$, a contradiction to (\ref{first:1}). We have established (\ref{first:main}).
\end{subequations}

\begin{subequations}\label{eq:see-co-triangle}
Next, we will show
\begin{equation}\tag{\ref{eq:see-co-triangle}}
 \begin{minipage}{0.9\linewidth}

For any co-triangle $C$ of $G - A$, some $v_i \in A$ is $C$-complete. 
\end{minipage}\end{equation}
Suppose that (\ref{eq:see-co-triangle}) is false for some co-triangle $C$ of $G - A$ with vertices $t_1, t_2, t_3$.
By (\ref{first:main}), any vertex $v_i$ of $A$ is adjacent to exactly two vertices of $C$.
\begin{equation}
 \begin{minipage}{0.9\linewidth}
\label{eq:different-neighborhood}
For any subscript $i$, the vertices $v_i, v_{i+1}$ cannot have the same neighbourhood in $C$,
\end{minipage}\end{equation}
for otherwise
$v_i, v_{i+1}$ form a co-triangle with their unique non-neighbour in $C$, a contradiction to the hypothesis of the Lemma.

Consider the vertices $v_1, v_2, v_3$. Suppose all of them have
different neighbourhoods in $C$. We may assume without loss of
generality that $N(v_1) \cap C =\{t_1, t_2 \}, N(v_2) \cap C
=\{t_1, t_3 \}, N(v_3) \cap C =\{t_2, t_3 \}$. Now $G[\{v_1, t_1,
v_2, t_3, v_3\}]$ induces a $C_5$. Thus, the  vertices $v_1, v_2,
v_3$ cannot all have distinct neighborhoods in $C$. It follows from
(\ref{eq:different-neighborhood}) that, for any $i$, $v_i$ and
$v_{i+2}$ have the same neighbourhood in $C$. Since $A$ has an odd
number of vertices, all vertices in $A$ have the same
neighbourhood in $C$. But this is a contradiction to
(\ref{eq:different-neighborhood}). We have established
(\ref{eq:see-co-triangle}).
\end{subequations}

Let $T$ be the set of vertices belonging to a co-triangle of $G$.
By (\ref{first:main}), we have $A \cap T = \emptyset$.
Write $R = V(G) - (T \cup A )$. We will establish the following.
\begin{equation}
 \begin{minipage}{0.9\linewidth}
\label{eq:see-co-triangle2} Let $C$ be a co-triangle of $G$, and
let $v,x$ be two vertices in $G-T$. Suppose $v$ is adjacent to all
vertices of $C$. If $x$ is not adjacent to $v$,  then $x$ is
adjacent to all vertices of $C$.
\end{minipage}\end{equation}
Let the vertices of the co-triangle $C$ be $t_1, t_2, t_3$.
Suppose $x$ is non-adjacent to $v$ and $t_3$.
Since $x \not\in T$, it is adjacent to $t_1, t_2$.
But now  $G[\{v,x,t_1, t_2, t_3\}]$ induces a banner in $G$.
We have established (\ref{eq:see-co-triangle2}).

Note (\ref{eq:see-co-triangle}) and (\ref{eq:see-co-triangle2})
imply each vertex of $A$ is adjacent to all vertices of $T$.
Furthermore, every vertex in $G-T$ that has a non-neighbour in $A$
is adjacent to all vertices of $T$. Now, consider the complement
$\overline{G}$ of $G$ and the component $O$ of $\overline{G}$ that
contains all of $A$. We claim that $O$ contains no vertex of $T$,
in particular, $\overline{G}$ is disconnected  and $O$ is
triangle-free. Suppose in $\overline{G}$, there is a path with one
endpoint in $A$ and the other endpoint in $T$. Take such a
shortest path $P$ and enumerate the vertices of the path  as $p_1,
p_2, \ldots, p_j$ with $p_1 \in A, p_j \in T$ and $p_\ell \in V(G)
- (T \cup A)$ for $\ell = 2, \ldots, j -1$. By
(\ref{eq:see-co-triangle2}), $p_\ell$ is adjacent, in $G$, to all
of $T$ for $\ell = 2, \ldots, j -1$; but this contradict the
existence of the non-edge $p_{j-1} p_j$ (in $G$) of $P$.
\end{proof}
\begin{lemma}\label{lem:co-triangle}
Let $G$ be a (banner, $C_5$)-free graph with an odd antihole $A$.
Suppose some two vertices of $A$ belong to a co-triangle.  Then
there is a homogeneous set $H$ of $G$ such  that $H$ contains $A$
and is co-triangle-free.
\end{lemma}
\begin{proof} 
Let $G$ and $A$ be defined as in the Lemma. Let $Q$ be the set of
vertices $q$ such that $q$ forms a co-triangle with some two
vertices of $A$.  The set $Q$ is not empty by the hypothesis of
the Lemma. Enumerate the vertices
of $A$ as $v_1, v_2 , \ldots , v_k$ such that $v_i v_{i+1}$ is a
non-edge of $G$ with the subscripts taken module $k$, with $k \geq 7$. We will
first establish the claim below.
\begin{equation}
 \begin{minipage}{0.9\linewidth}
\label{eq:no-neighbor-in-A}
Every vertex in $Q$ is $A$-anticomplete.
\end{minipage}\end{equation}
Suppose some vertex $u \in Q$ has some neighbour in $A$. Since $u$ is non-adjacent to some two consecutive
vertices of the antihole, we may assume $u v_1, u v_2 \not \in E(G),$ and $ u v_3 \in E(G)$. We have
the following implications:
\newline \hspace*{2em}
   $u v_4  \not\in E(G)$, for otherwise $G[\{v_1, v_2, v_3, v_4, u \}]$ induces a banner,
\newline \hspace*{2em}
   $u v_5 \not\in E(G)$, for otherwise $G[\{v_1, v_2, v_4, v_5, u \}]$ induces a banner,
\newline \hspace*{2em}
   $u v_6 \in E(G)$, for otherwise $G[\{v_2, v_3, v_5, v_6, u \}]$ induces a banner.

But now $G[\{v_1, v_2, v_5, v_6, u \}]$ induces a banner. We have established (\ref{eq:no-neighbor-in-A}).

Let $R$ be the set of vertices in $V(G) - (A \cup Q)$ that have some neighbors in $Q$, and let $S$
be the set of vertices in $V(G) - (A \cup Q)$ that  have no neighbors in $Q$.
Note that $V(G) = A \cup Q \cup R \cup S$. We will show that
\begin{equation}
 \begin{minipage}{0.9\linewidth}
\label{eq:two-consecutive}
Each vertex in $R \cup S$ is adjacent to some two consecutive vertices of $A$.
\end{minipage}\end{equation}
Consider a vertex $x \in R \cup S$. Since $x \not\in Q$, for any two consecutive vertices of $A$, the vertex $x$ must be adjacent to at least one of them. Since $A$ has an odd number of vertices, $x$ must be adjacent to some two consecutive vertices of $A$.
So (\ref{eq:two-consecutive}) holds . Next, we show that
\begin{equation}
 \begin{minipage}{0.9\linewidth}
\label{eq:see-all-A}
Each vertex in $R$ is $A$-complete.
\end{minipage}\end{equation}
Consider a vertex $x$ in $R$. It is adjacent to some vertex $u$ in $Q$.
By (\ref{eq:two-consecutive}), $x$ is adjacent to some two consecutive vertices of $A$, say, $v_1, v_2$. For any
$j \in \{4, 5, \ldots, k-1\}$, $x$ is adjacent to $v_j$, for otherwise, by (\ref{eq:no-neighbor-in-A}),  $G[\{v_j, v_1, x, v_2, u\}]$
induces a banner. Now, $x$ is adjacent to $v_3$, for otherwise $G[\{v_{k-1}, v_{k-2}, v_3, x,u\}]$
induces a banner. Similarly, $x$ is adjacent to $v_k$. We have establish (\ref{eq:see-all-A}).
\begin{equation}
 \begin{minipage}{0.9\linewidth}
\label{eq:see-all-S}
Each vertex in $R$ is  $S$-complete.
\end{minipage}\end{equation}
Suppose there are vertices $r \in R, s \in S$ with $rs \not\in
E(G)$. By  (\ref{eq:two-consecutive}), $s$ is adjacent to some two
consecutive vertices $v_i, v_{i+1}$ of $A$. By
(\ref{eq:see-all-A}), $r$ is adjacent to $v_i, v_{i+1}$. But now
$G[\{v_i, v_{i+1}, r, s, u\}]$ induces a banner for some neighbor
$u$ in $Q$ of $r$. Thus (\ref{eq:see-all-S}) holds. 

Write $G' = G[A \cup S]$. Then $G'$ satisfies the hypothesis of Lemma~\ref{lem:triangle-free-component}. Let $O$ be the 
component of $\overline{G'}$ that contains $A$. Lemma~\ref{lem:triangle-free-component} implies $O$ is a homogeneous set of $G'$ and is co-triangle-free. By (\ref{eq:no-neighbor-in-A}), (\ref{eq:see-all-A}), and (\ref{eq:see-all-S}), every vertex of $Q \cup R$ is either $O$-complete or $O$-anticomplete, $O$ is a a homogeneous set of $G$. 
\end{proof}
\begin{theorem}\label{thm:main}
Let $G$ be a (banner, odd hole)-free graph. Then one of the
following holds for $G$.
\begin{description}
 \item[(i)] $G$ is Berge.
 \item[(ii)] $\alpha(G) \leq 2$.
 \item[(iii)] Every odd antihole $A$ of $G$ is contained in a homogeneous set $H$ of $G$ such that
           $G[H]$ is co-triangle-free.
\end{description}
\end{theorem}
\begin{proof}
Let $G$ be a (banner, odd hole)-free graph, and assume that (i) and (ii) do not hold.
Thus, $G$ contains an odd antihole $A$.
By Lemmas~\ref{lem:triangle-free-component} and~\ref{lem:co-triangle}, there is  
a homogeneous $H$ set of $G$ that contains $A$ and is co-triangle-free. 
\end{proof}
We note the homogeneous set in (iii) can be easily found in $O(n^3)$ if an anti-hole is given.
Actually, with a bit more work, one can find it in $O(n^2)$ time (given
the odd antihole), but this step is not the bottle neck of our recognition and coloring algorithms.

\section{Recognition algorithms}\label{sec:recognition}
In this section we give two polynomial-time algorithms for recognizing a
(banner, odd hole)-free graph. Both algorithms rely on the following easy observation.
\begin{observation}\label{obs:homogeneous-set}
Let $G$ be a graph with a homogeneous set $H$. Let $P$ be a prime induced subgraph of $G$.
Then $P$ cannot contain two vertices in $H$ and some vertex in in $G-H$. \qed
\end{observation}
The following statement
follows from Observation~\ref{obs:homogeneous-set} and the fact that the odd holes are prime.
\begin{observation}\label{obs:odd-hole}
Let $G$ be a banner-free graph with a homogeneous set $H$. Then
$G$ contains no odd hole if and only if both $H$ and $G[(V(G) - H) \cup \{h\}]$ contain no odd hole,
for any vertex $h \in H$. \qed
\end{observation}
Let ${\cal B}$ denote the class of (banner, odd hole)-free
graphs. The first algorithm follows the steps suggested by
Theorem~\ref{thm:main}. First, it checks in $O(n^5)$ time that the
input graph $G$ is (banner, $C_5$)-free.  If $\alpha(G) \leq 2$,
then $G$ is in ${\cal B}$. Assume $\alpha(G) \geq 3$. Also we may
assume without loss of generality that $G$ is connected. Now, use
the Berge graph recognition algorithm to find an odd hole, or
odd antihole (if one exists) in $O(n^{10})$ time (it is $O(n^9)$ to recognize Berge graphs,
but to find an odd hole or odd antihole, we have to repeat this algorithm $n$ times). 
If no such hole or antihole is
found, $G$ is in ${\cal B}$. If an odd hole is found, $G$ is not
in ${\cal B}$. Assume an odd antihole $A$ is found. Using the proof of 
Theorem~\ref{thm:main}, find a homogeneous set $H$ with $\alpha(H) =2 $ containing $A$
in $O(n^3)$ time. Recursively apply the above procedure on 
$G[(V(G) - H) \cup \{h\}]$ for a vertex $h \in H$ (since $\alpha(H) = 2$, $H$ contains no odd hole and so we do not 
need to apply the procedure on $H$).  The recursion is called
at most $n$ times by the analysis of Algorithm~\ref{al:modular-decomposition}. Thus, our algorithm runs in $O(n^{11})$ time.

The second algorithm bypasses the obvious approach of the first algorithm. It uses deep properties of the
modular decomposition discussed in Subsection~\ref{sec:modular-decomposition}. Observation~\ref{obs:homogeneous-set}
implies the following Lemma.
\begin{lemma}\label{lem:recognition1}
Let $G$ be a banner-free graph with a modular decomposition ${\cal P} = (M_1, M_2,$  $\ldots,  \, M_\ell)$.
Then $G$ is odd-hole-free if and only if the quotient graph $G_{ {\cal P} }$ and all graphs $G[M_i]$ are odd-hole-free. \qed
\end{lemma}
Given a graph $G$. The second algorithm starts by verifying in $O(n^5)$ time that $G$ is (banner, $C_5$)-free. If $G$ is not connected, then we recursively apply our algorithm on each component $C_i$ of $G$; graph $G$ is in ${\cal B}$ if and only if each $G[C_i]$ is in ${\cal B}$. If $\overline{G}$ is disconnected, then we recursively apply our algorithm on each subgraph $G_i$ of $G$ induced by the vertices of the components of $\overline{G}$; graph $G$ is in ${\cal B}$ if and only if each $G_i$ is in ${\cal B}$. Now, suppose that both $G$ and $\overline{G}$ are connected. 
If $\alpha(G) \leq 2$, then $G$ is in ${\cal B}$ since $G$ cannot contain an odd hole of length at least seven.  
Assume $\alpha(G) \geq 3$. The algorithm finds a modular decomposition of $G$ with the partition ${\cal P} = (M_1, M_2,$  $\ldots,  \, M_\ell)$.
The quotient graph $G _{ {\cal P}}$ is prime. If  $\alpha(G _{ {\cal P}}) \leq 2$, then $G _{ {\cal P}}$ is in ${\cal B}$ since $G$ cannot contain an odd hole of length at least seven. Consider the case where $\alpha(G _{ {\cal P}}) > 2$. 
By Theorem~\ref{thm:main}, $G _{ {\cal P}}$ is  in ${\cal B}$ if and only if $G _{ {\cal P}}$ is Berge. Bergeness of $G _{ {\cal P}}$ can be verified in $O(n^9)$ time using the 
algorithm of \cite{ChuCor2005}.
Now, recursively verify that every $G[M_i]$ is odd-hole-free. Theorem~\ref{thm:decomp} implies $ \mid G _{ {\cal P}} \mid + \sum_1^\ell$
$\mid M_i \mid = O(n)$. Thus, the second algorithm runs in $O(n^9)$ time.

Actually, we have shown the following: Let $p(n)$ (respectively, $c(n)$, $b(n)$, $t(n)$) be the time complexity of
finding an odd hole or odd antihole (respectively, $C_5$, banner, triangle), and let $a(n) = max(p(n), c(n), b(n), t(n))$. Then there is
an $O(a(n))$-time algorithm to recognize graphs in ${\cal B}$. Here, we are making the reasonable assumption that $a(n)$ is at least
$O(n+m)$. In the current state of knowledge, recognizing perfect graph is the bottleneck of the algorithm.

\section{Optimizing (banner, odd hole)-free graphs}\label{sec:optimization}
In this section, we consider the following four optimization
problems: finding a minimum coloring, finding a minimum clique
cover (the minimum coloring in the complement of the graph),
finding a largest stable set, and finding a largest clique. It is
interesting to note that for our class of (banner, odd hole)-free
graphs, the coloring and largest stable set problems are solvable
in polynomial time, whereas the clique cover and largest clique
problems are NP-hard.
\subsection{Finding a minimum coloring in polynomial time}\label{sec:coloring}
In this section, we describe a robust algorithms to  color a
(banner, odd hole)-free graph. The polynomial-time perfect graph coloring
algorithm of Gr\"otschel et al. (\cite{GroLov1984}) is robust.
Given a graph $G$, the algorithm from \cite{GroLov1984} either returns an optimal
coloring of $G$ (and a clique of the same size proving the
coloring is minimum) or a correct declaration that $G$ is not
perfect.

Concerning graphs $G$ with $\alpha(G) \leq 2$, it is well known
that a minimum coloring of  $G$ can be found by finding a maximum
matching $M$ in the complement $\overline{G}$ of $G$. (Let $t$ be
the number of edges of $M$. Then we have $\chi(G) = n - t$.)  The
algorithm of Micali and Vazirani (\cite{MicVaz1980}) finds a
maximum matching of a graph in $O(m \sqrt{n})$ time.

We may assume the algorithm below (Algorithm~\ref{alg:easy-color}:
DIRECT-COLOR(G)) runs in polynomial time. DIRECT-COLOR($G$)
returns an optimal (minimum) coloring if $G$ is perfect or satisfies
$\alpha(G) \leq 2$.
\begin{algorithm}
\caption{DIRECT-COLOR(G)}\label{alg:easy-color}
\begin{algorithmic}
 \STATE{\textbf{input:}} graph $G$
 \STATE{\textbf{output:}} An optimal coloring of $G$ or a message that $G$ is imperfect with $\alpha(G) \geq 3$.
 \STATE{\textbf{remark:}} The algorithm returns an optimal coloring if $G$ is perfect or $\alpha(G) \leq 2$.
 \STATE
 \IF{ $\alpha(G) \leq 2$}
    \STATE{find an optimal coloring of $G$ via the matching algorithm}
    \STATE{return the optimal coloring and stop}
 \ENDIF
 \STATE
 \STATE{apply the perfect graph coloring algorithm on $G$}
 \IF{an optimal coloring is returned}
    \STATE{return the optimal coloring and stop}
    \ELSE
    \STATE{return ``$G$ is imperfect with $\alpha(G) \geq 3$''}
 \ENDIF
\end{algorithmic}
\end{algorithm}

Let $G$ and $L$ be two vertex-disjoint graphs. Let $G$ have a
homogeneous set $H$. Define $g(G, H, L)$ to be the graph obtained
from $G$ by substituting $L$ for $H$, that is, $g(G, H, L)$ is
obtained from $G$ by (i) removing the set $H$, (ii) adding the
graph $L$, and (iii) for each vertex $v \in G-H$ that has a
neighbor in $H$, adding all edges between $v$ and $L$. Note that
$L$ is a homogeneous set of $g(G, H, L)$. Since the following
observation is easy to establish, we omit its proof.
\begin{observation}\label{obs:homegeneous-coloring}
Let $G$ be a graph with a homogeneous set $H$. Let $\chi(H) = k$. Let $K_k$ be a clique
on $k$ vertices. Then
\begin{description}
\item[(i)]
   $\chi(G) = \chi(\; g(G,H,K_k) \;)$.
   \item[(ii)] if $G$ is (banner, odd hole)-free then so is $g(G,H,K_k)$.  \qed
\end{description}
\end{observation}
We need to discuss part (i) of  Observation~\ref{obs:homegeneous-coloring} before describing our algorithm.
Assume we are given a minimum coloring of $H$ with stable sets (color classes) 
$S_1, \ldots, S_k$.
Enumerate the vertices of $K_k$ as $v_1, \ldots, v_k$. Assume in a minimum coloring of $g(G,H,K_k)$, the vertices
$v_i$ has color $c(i)$. From  a minimum coloring of $g(G,H,K_k)$, we may obtain a minimum coloring of $G$ by
assigning to each vertex in $S_i$ the color $c(i)$.

Consider a graph $G$ with a modular decomposition ${\cal P} = (M_1, \ldots , M_t)$. Note that
if $G$ is prime then each $M_i$ is a single vertex. Our coloring algorithm (Algorithm~\ref{alg:color}: COLOR($G$)) will return a minimum coloring
of $G$, or correctly declare that $G$ is not (banner, odd hole)-free. If $G$ is prime, then a single call to
DIRECT-COLOR($G$) will return a minimum coloring of $G$, or a message that $G$ is imperfect with $\alpha(G) \geq 3$. In  the latter case, by Theorem~\ref{thm:main} we know that $G$ is not (banner, odd hole)-free.
Now, assume $G$ is not prime.
For each module $M_i$, the algorithm
is recursively applied to $M_i$. If $M_i$ is not (banner, odd hole)-free, then neither is $G$. If a minimum
coloring of $M_i$ with $t$ colors is returned, then we substitute a clique with $t$ vertices for $M_i$ in $G$.
This substitution is done for all modules of the decomposition. Let $F$ be the graph obtained from these
substitutions. By Theorem~\ref{thm:main}, if $\alpha(F) \geq 3$, then $F$ cannot contain an odd antihole $A$, for
otherwise  $A$ is contained in some strong module of $G$ and this module overlaps some strong  module $M_i$
of the decomposition, a contradiction. Thus a single call to DIRECT-COLOR($F$) will return a minimum coloring of $F$ (and thus,
of $G$) or a correct declaration that $G$ is not (banner, odd hole)-free.
Our algorithm is described below (Algorithm~\ref{alg:color}).
\begin{algorithm}
\caption{COLOR(G)}\label{alg:color}
\begin{algorithmic}
 \STATE{\textbf{input:}} graph $G$ with a modular decomposition ${\cal P} = (M_1, \ldots, M_\ell)$.
 \STATE{\textbf{output:}} an optimal coloring of $G$ or a message that $G$ is not (banner, odd hole)-free.
 \STATE
 \IF{$G$ is prime}
    \STATE call DIRECT-COLOR($G$)
    \IF{a minimum coloring ${\cal C}$ of $G$ is returned}
             \STATE return  ${\cal C}$ and stop
         \ELSE
            \STATE return ``$G$ is not (banner, odd hole)-free'' and stop
    \ENDIF
 \ENDIF
 \STATE
 \FOR{each $M_i$ of ${\cal P} $}
   \STATE call COLOR($M_i$)
   \IF{an optimal coloring of $M_i$ with $t_i$ colors is returned}
      \STATE{construct a clique $C_i$ on $t_i$ vertices}
    \ELSE
      \STATE{return ``$G$ is not (banner, odd hole)-free'' and stop}
    \ENDIF
 \ENDFOR
 \STATE
  \STATE{let $F$ be the graph obtained from $G$ by substituting the clique $C_i$ for $M_i$ for all $i$}
  \STATE{call DIRECT-COLOR($F$)}
      \IF{an optimal coloring ${\cal C}$ of $F$  is returned}
       \STATE{ from ${\cal C}$, construct an optimal coloring ${\cal C}'$ of $G$}
       \STATE return ${\cal C}'$ and stop
     \ELSE
        \STATE return ``$G$ is not (banner, odd hole)-free'' and stop
     \ENDIF

\end{algorithmic}
\end{algorithm}

The algorithm can be make more efficient by skipping the trivial 
modules of the decomposition, but this does not improve the worst
case complexity. Since the modules of ${\cal P}$ are strong and
therefore pairwise non-intersecting, the number of  recursively
calls to COLOR is bounded by $2n -1$ for $n \geq 2$. This fact can be easily seen
by induction as follows: let $s_i = | M_i |$, by induction the
number of recursive calls on $G$ (with the modules $M_i$) is at most $ 1 + (2 s_1
- 1) + (2 s_2 - 1) + \ldots + (2 s_t -1 ) = (2
\sum_{i=1}^{t} s_i) - t + 1 \leq  2n -1$. Let $g(n)$
(respectively, $t(n)$) be the time complexity of coloring a
perfect graph (respectively, find a maximum matching of a graph).
By Theorem~\ref{thm:decomp},  Algorithm~\ref{alg:color} has
complexity $O( n \; (max(g(n), t(n) ) )$.
\subsection{Finding a largest clique is hard}\label{sec:find-clique}
It was proved by Poljak (\cite{Pol1974}) that it is NP-hard to
determine $\alpha(G)$ for a triangle-free graph $G$. We are going to use Poljak's argument to show that 
finding a largest clique of a (banner, odd hole)-free graph is NP-hard. Let $G$ be a graph.
Construct a graph $f(G)$ from $G$ by, for each edge $ab$ of $G$,
replacing $ab$ by an induced path on three edges, that is, we
subdivide the edge $ab$ twice. Then it is easy to see that
$\alpha(f(G)) = \alpha(G) + m$, where $m$ is the number of edges
of $G$. The graph $f(G)$ is triangle-free. By repeatedly applying
this construction, it is seen that it is NP-hard to compute
$\alpha(G)$ for a graph $G$ of any given minimum girth $g$. In particular,
it is NP-hard to compute $\alpha(G)$ for a (triangle, $C_5$)-free
graph $G$, or equivalently, to compute $\omega(F)$ for a
(co-triangle, $C_5$)-free graph $F$. Since (co-triangle,
$C_5$)-free graphs are (banner, odd hole)-free, it is NP-hard to
compute $\omega(G)$ for a (banner, odd hole)-free graph.
\subsection{Finding a minimum clique cover is hard}\label{sec:find-clique-cover}
It was noted by Jensen and  Toft (section 10.3 of
\cite{JenTof1995}) that it is NP-complete to decide, for any fixed
integer $g$, whether a graph $G$ of girth at least $g$ is
6-colorable (this result uses the Haj\'os construction for graphs
of high chromatic number). Thus, it is NP-hard to find a minimum
coloring of a $(C_3, C_5)$-free graph, or equivalently, to
find a minimum clique cover of a (co-triangle,  $C_5$)-free
graph $G$. Since such a graph $G$ is (banner, odd hole)-free, it
is NP-hard to find a minimum clique cover of a (banner, odd
hole)-free graph.
\subsection{Finding a largest stable set in polynomial time}\label{sec:find-stable-set}
In this section, we will describe a polynomial-time algorithm for
find a largest stable set of a (banner, odd hole)-free graph. Our
algorithm is robust. We will actually solve the more general
problem of finding a maximum weighted stable set.

Let $G$ be a graph whose vertices are given ``weights'', that is,
each vertex $v$ is given an integer $w_G(v)$. When the context is
clear, we will drop the subscript $G$ and let $w(x) = w_G(x)$. The
problem is to find a stable set $S$ which maximizes $\sum_{v \in
S} w_G(v)$; this sum is denoted by $\alpha_w(G)$. We will refer to
this problem as the Maximum Weighted Stable Set problem, or WSS
for short. The algorithm of Gr\"otschel et al. (\cite{GroLov1984})
actually solves, for a perfect graph, the ``minimum weighted
coloring'' (which we will define later) and the ``maximum
weighted clique'' (which is WSS for the complement of the graph).
The algorithm of Gr\"otschel et al. is a robust algorithm: given a
weighted graph $G$, it either returns a maximum weighted stable
set of $G$, or correctly declares that $G$ is not perfect.
Finally, when $G$ is a graph with $\alpha(G) \leq 2$, then WSS
can be solved in polynomial time since there are only $O(n^2)$
stable sets. We can list them all and take the one with the
maximum weight.

Let $H$ be a homogeneous set of a weighted graph $G$. By $g(G,H,h)$, we
denote the weighted graph obtained from $G$ by substituting a
vertex $h$ for $H$ where the weight function $w'$ for $g(G,H,h)$
is defined as follows.  For the vertex
$h$, we set $w'(h) = \alpha_w(H)$ and for all $x \in G-H$, we set $w'(x) = w(x)$.  
It is easy to see that
$\alpha_w(G) = \alpha_{w'}( g(G,H,h) )$.

Let $G$ be a connected (banner, odd hole)-free graph with
$\alpha(G) \geq 3$. If $G$ is not perfect, then by
Theorem~\ref{thm:main}, $G$ contains a homogeneous set $H$. We
recursively solve WSS for $H$, then substitute a vertex $h$ for
$H$, where the weight of $h$ is equal to $\alpha_w(H)$. We obtain
the graph $F=g(G,H,h)$. Now, solve WSS for $F$. From a maximum
weighted stable set of $F$, we can easily construct a maximum
weighted stable set of $G$. This discussion shows WSS can be
solved in polynomial time for (banner, odd hole)-free graphs. We
can design a more efficient, and robust, algorithm by imitating
the coloring algorithm in Section~\ref{sec:coloring}.

First, we construct the algorithm DIRECT-STABLE-SET
(Algorithm~\ref{alg:easy-stable-set}) that takes as input a
weighted graph $G$. If $\alpha(G) \leq 2$, it returns a maximum
weighted stable set and terminates. Otherwise, it applies the
stable set algorithm of Gr\"otschel et al. to $G$. Then the
algorithm either returns a maximum weighted stable set, or a
message that $G$ is imperfect with $\alpha(G) \geq 3$. 
\begin{algorithm}
\caption{DIRECT-STABLE-SET(G)}\label{alg:easy-stable-set}
\begin{algorithmic}
 \STATE{\textbf{input:}} a graph $G$ with a weight function $w$ on its vertices.
 \STATE{\textbf{output:}} a maximum weighted stable set of $G$  or a message that $G$ is imperfect with $\alpha(G) \geq 3$.
 \STATE{\textbf{remark:}} The algorithm returns a maximum weighted stable set of $G$ whenever $G$ is perfect or $\alpha(G) \leq 2$.
 \STATE
 \IF{ $\alpha(G) \leq 2$}
    \STATE{find maximum weighted stable set $S$ of $G$ by listing all stable sets of $G$ and taking one with the largest weight}
    \STATE{return $S$ and stop}
 \ENDIF
 \STATE
 \STATE{apply the stable set algorithm for perfect graph  on $G$}
 \IF{a maximum weighted stable set $S$ is returned}
    \STATE{return $S$ and stop}
    \ELSE
    \STATE{return ``$G$ is imperfect with $\alpha(G) \geq 3$''}
 \ENDIF
\end{algorithmic}
\end{algorithm}

Consider a graph $G$ with a modular decomposition ${\cal P} =
(M_1, \ldots, M_\ell )$. Algorithm~\ref{alg:mwss} (WSS($G$)) returns a maximum weighted stable
set of $G$, or correctly declares that $G$ is not (banner, odd
hole)-free. If $G$ is prime, then a single call to
DIRECT-STABLE-SET($G$) either returns a maximum weighted stable
set of $G$, or a message that $G$ is imperfect with $\alpha(G) \geq 3$. In the latter case, 
by Theorem~\ref{thm:main}, we know $G$ is not (banner, odd hole)-free.

Now, suppose $G$ is not prime. If $G$ is not connected, then we can apply our algorithm
on each component $C_i$ of $G$, and it is easy to construct a maximum weighted stable set of $G$ from 
the maximum weighted stable sets of the components $C_i$. Similarly, if $\overline{G}$ is disconnected, then
we apply our algorithm on each component of $\overline{G}$. The details of these two cases are spelled out in
Algorithms~\ref{alg:union} and~\ref{alg:join} in the Appendix.

Now, we may assume that $G$ and $\overline{G}$ are connected.
For each module $M_i$ of the modular decomposition,
the algorithm is recursively applied on $G[M_i]$. If $G[M_i]$ is not
(banner, odd hole)-free, then neither is $G$. If a maximum
weighted stable set $S_i$ is returned, then the algorithms
substitutes a vertex $h_i$ for $M_i$ in $G$ and gives $h_i$ a
weight equal to $\alpha_w(M_i)$. This substitution is done for all
modules of the decomposition. Let $F$ be the graph obtained from
these substitutions. Note that $F$ is the quotient graph $G_{
{\cal P} }$ and so it is prime. By Theorem~\ref{thm:main}, if $\alpha(F) \geq 3$, then $F$ cannot contain an odd antihole $A$. Thus a single call to DIRECT-STABLE-SET($F$) will return a maximum weighted stable set of $F$ (from this
we can construct a maximum weighted stable set of $G$ in the
obvious way) or a correct declaration that $G$ is not (banner, odd
hole)-free. Our algorithm is described below
(Algorithm~\ref{alg:mwss}). Let $p(n)$ be the time complexity of
finding a maximum weighted stable set of a perfect graph. Our
algorithm runs in $O( p(n) + n^2 )$ time with the term $n^2$
coming from solving WSS for graph with no co-triangles.

\begin{algorithm}
\caption{WSS(G)}\label{alg:mwss}
\begin{algorithmic}
 \STATE{\textbf{input:}} graph $G$ with  a modular decomposition ${\cal P} = (M_1, \ldots, M_\ell)$ and a weight function $w$ on its vertices.
 \STATE{\textbf{output:}} a maximum weighted stable set of $G$ or a message that $G$ is not (banner, odd hole)-free.
 \STATE
\IF{$G$ is prime}
    \STATE call DIRECT-STABLE-SET($G$)
    \IF{a maximum weighted stable set $S$ of $G$ is returned}
             \STATE return  $S$ and stop
         \ELSE
            \STATE return ``$G$ is not (banner, odd hole)-free'' and stop
    \ENDIF
 \ENDIF
 \STATE
 \IF{$G$ is not connected}
	 \STATE call STABLE-SET-UNION($G$)
	 \IF{ a maximum weighted stable set $S$ is returned }
		 \STATE return $S$ and stop
	\ELSE
		\STATE return 	``$G$ is not (banner, odd hole)-free'' and stop
	\ENDIF
\ENDIF
\STATE
 \IF{$\overline{G}$ is not connected}
	 \STATE call STABLE-SET-JOIN($G$)
	\IF{ a maximum weighted stable set $S$ is returned }
		\STATE return $S$ and stop
	\ELSE
		\STATE return 	``$G$ is not (banner, odd hole)-free'' and stop
	\ENDIF
\ENDIF
 \STATE
 \FOR{each module $M_i$ of ${\cal P} $}
   \STATE call WSS($G[M_i]$)
   \IF{a maximum weighted stable set $S_i$ of $G[M_i]$  is returned}
      \STATE{compute $\alpha_w(G[M_i])$ which is the sum of the weights of the vertices in $S_i$}
    \ELSE
      \STATE{return ``$G$ is not (banner, odd hole)-free'' and stop}
    \ENDIF
 \ENDFOR
 \STATE
  \STATE{Let $F$ be the graph obtained from $G$ by substituting a vertex $s_i$ for $M_i$ for all $i$, and set $w(s_i) = \alpha_w(G[M_i])$}
  \STATE{call DIRECT-STABLE-SET($F$)}
      \IF{a maximum weighted stable set $S$ of $F$  is returned}
       \STATE{ construct a maximum weighted stable set $S^*$ of $G$ from $S$}
       \STATE{ return $S^*$ and stop}
     \ELSE
        \STATE return ``$G$ is not (banner, odd hole)-free'' and stop
     \ENDIF
 \end{algorithmic}
 \end{algorithm}

\section{The 2-divisibility
of (banner, odd hole)-free graphs}\label{sec:divisibility}
In this section, we prove that (banner, odd hole)-free graphs are
2-divisible. Recall that a graph $G$ is {\it $k$-divisible} if the
vertex set of each induced subgraph $H$ of $G$ with at least one
edge can be partitioned into $k$ sets none of which contains a
clique of size $\omega(H)$.  It is easy to see that perfect graphs
are 2-divisible. We will need the following two results by Ho\`ang
and McDiarmid (\cite{HoaMcd2002}).
\begin{lemma}[\cite{HoaMcd2002}, Lemma 3]\label{lem:substitution}
Let $G$ be a  graph with a homogeneous set $H$. If for some $k
\geq 2$, both  $H$ and $G[(V(G) - H) \cup \{h\}]$, for any vertex
$h \in H$, are $k$-divisible, then so is $G$. \qed
\end{lemma}
\begin{theorem}[\cite{HoaMcd2002}, Theorem 1] \label{thm:claw-free-2div}
A (claw, odd hole)-free graph is 2-divisible. \qed
\end{theorem}
\begin{theorem}\label{thm:2-divisible}
A (banner, odd hole)-free graph is 2-divisible.
\end{theorem}
\begin{proof}
We prove the theorem by induction on the number of vertices. Let
$G$ be a (banner, odd hole)-free graph. By the induction
hypothesis, we may assume every proper induced subgraph of $G$ is
2-divisible. We may assume that $G$ is not perfect, and by
Theorem~\ref{thm:claw-free-2div}, $G$ has $\alpha(G) \geq 3$. By the Strong Perfect Graph Theorem,
$G$ contains an odd antithole with at least seven vertices.  By
Theorem~\ref{thm:main}, $G$ contains a homogeneous set $H$. By the
induction hypothesis, both  $H$ and $G[(V(G) - H) \cup \{h\}]$
(for any vertex $h \in H$) are 2-divisible. Now by
Lemma~\ref{lem:substitution}, $G$ is 2-divisible.
\end{proof}
As noted in Section~\ref{sec:introduction},
Theorem~\ref{thm:2-divisible} implies that a (banner, odd
hole)-free graph $G$ has $\chi(G) \leq 2^{\omega(G) -1}$. In
Section~\ref{sec:perfect-divisibility}, we will obtain a better
bound on the chromatic number of such graphs. 
We note the result of Scott and Seymour (\cite{ScoSey2015})
showing that an odd-hole-free graph $G$ has $\chi(G) \leq
2^{2^{\omega(G)+2}}$.
\section{The perfect divisibility
of (banner, odd hole)-free graphs}\label{sec:perfect-divisibility}
Recall that a graph $G$ is {\it perfectly divisible} if every induced
subgraph $H$ of $G$ contains a set $X$ of vertices such that $X$
meets all largest cliques of $H$, and $X$ induces a perfect graph.
Such a set $X$ will be called a {\it compact} set.  We start with
the following easy observation.
\begin{observation}\label{obs:PD-chi-bounded}
A perfectly divisible graph $G$ satisfies $\chi(G) \leq \omega(G)^2$.
\qed
\end{observation}
In this section, we will prove (banner, odd hole)-free graphs are
perfectly divisible. Let us say that a graph $G$ is {\it minimally
non-perfectly divisible} if $G$ is not perfectly divisible but
every proper induced subgraph of $G$ is.
\begin{observation}\label{obs:MNPD-homogeneous}
No minimal non-perfectly divisible graph contains a homogeneous set. \qed
\end{observation}
We will prove a statement more general than Observation~\ref{obs:MNPD-homogeneous}. In order to do this, we will need to work with weighted graphs.  Consider
the following two optimization problems.

\noindent {\bf Minimum Weighted Coloring (WCOL)} Given a weighted graph
$G$ such that each vertex $x$ has a weight $w(x)$ which is a
positive integer. Find stable sets $S_1,S_2,\ldots,S_k$ and
integers $I(S_1), \ldots, I(S_k)$ such that for each vertex $x$ we
have $w(x) \leq \Sigma_{x \in S_i}I(S_i)$ and that the sum of the
numbers $I(S_i)$ is minimized. This sum is called the weighted
chromatic number and denoted by $\chi_w(G)$.

\noindent {\bf Maximum Weighted Clique (WCLI)} Given a weighted graph $G$
such that each vertex $x$ has a weight $w(x)$ which is a positive
integer. Find a clique $C$ such that $\Sigma_{x \in C}w(x)$ is
maximized. This sum is called the weighted clique number and
denoted by $\omega_w(G)$.

For perfect (weighted) graph $G$, it is known that $\chi_w(G) = \omega_w(G)$. This is a consequence of 
the Substitution Lemma (Lemma~\ref{lem:substitution} below) proved by Lov\'asz \cite{Lov1972}. He proved the following: Consider a perfect graph $G$ and a vertex $x$ of $G$. Then the graph obtained from $G$ by substituting another perfect graph  $H$ for $x$ is perfect. We will rephrase the Substitution Lemma as follows.  
\begin{lemma}\label{lem:substitution}
	No minimal imperfect graph contains a homogeneous set.
\end{lemma}

The reader may see an analogy between Observation~\ref{obs:MNPD-homogeneous} and Lemma~\ref{lem:substitution}. 
An unweighted graph can be turned into a weighted graph by assigning to each vertex a weight of one. Thus,
WCOL generalizes the coloring problem, and WCLI generalizes the clique problem. We may define the notion of perfect divisibility for weighted graphs. 
A weighted graph $G$ is {\it perfectly divisible} if each induced subgraph $H$ of $G$ contains a set 
$X$ such that $X$ meets all maximum weighted cliques of $H$ and $X$ induces a perfect graph. We may define analogously the notions of ``compact set'' and ``minimal non-perfectly divisible'' for weighted graphs. The following observation implies Observation~\ref{obs:MNPD-homogeneous}.
\begin{observation}\label{obs:weighted-homogeneous}
Let $G$ be a weighted graph. If $G$ is minimally non-perfectly divisible, then $G$ cannot contain a homogeneous set. 
\end{observation}
{\it Proof}. Let $G$ be a weighted graph. Suppose $G$ is minimally non-perfectly divisible, and contains a homogeneous set $H$. Let $G_h$ be the graph obtained
from $G$ by substituting a vertex $h$ for $H$  and  setting the weight of $h$ to $w(h) = \omega_w(H)$. (Thus $G_h$ is isomorphic to $G[V(G) - (H-v)]$ for any vertex $v \in H$.) It is easy to see 
that $\omega_w(G) = \omega_w(G_h)$. 

The minimality of $G$
implies that $G_h$ contains a compact set $C_1$ and $H$ contains a
compact set $C_2$. Suppose first that $h \in C_1$.
Define $C = (C_1 - h) \cup C_2$. Since both
$C_1$ and $C_2$ induce a perfect graph,  the graph $G[C]$ is
perfect by Lemma~\ref{lem:substitution} (if $|C_2| = 1$, then $G[C]$ is a proper induced subgraph of $G$, otherwise, $C_2$ is a homogeneous set of $G[C]$).
We will show that $C$ is compact.  Consider a largest weighted clique $K$ of $G$. If $K$
contains no vertex of $H$, then $K$ lies entirely in $G_h -h$. Since in $G_h$, $C_1 - h$ meets $K$, we know that $K$ is met by $C$. 
If $K$ contains some vertex of $H$,
then $K \cap H$ is a maximum weighted clique of $H$ and thus is met by
$C_2$. So, $C$ is compact, a contradiction to the minimality of
$G$. 

Now, we may assume that $h \not\in C_1$. We claim that $C_1$ meets all maximum weighted cliques of 
$G$, and thus is compact. Suppose $G$ has a maximum weighted clique $K$ such that $K \cap C_1 = \emptyset$. 
Then $K$ must necessarily contain some vertices of $H$, for otherwise $K$ lies entirely in $G_h -h$ and is met by $C_1$. Let $N$ be the
set of vertices in $G-H$ that have a neighbour in $H$.
Write $K_H = K \cap H$ and $K_N = K \cap N$. Note that $K = K_N \cup K_H$. Since the weight of $h$ is at least the weight of the clique $K_H$ (ie. $w(h) \geq \sum_{x \in K_H} w(x)$), $K_N \cup \{h\}$ is a maximum weighted clique of $G_h$, but is not met by $C_1$, a contradiction.
\qed  
\begin{observation}\label{obs:PD-alpha}
If a graph $G$ contains no co-triangles, then $G$ is perfectly
divisible.
\end{observation}
{\it Proof}.  We will prove the Observation by induction on the
number of vertices. Let $G$ be a graph without co-triangles. By
the induction hypothesis, we only need prove $G$ contains a
compact set. Consider a vertex $x$ of $G$. Let $N$ be the set of
vertices adjacent to $x$ and let $M= V(G)-N-\{x\}$. Since $G$
contains no co-triangles, $M$ is a clique. Write $C= M \cup
\{x\}$. Then $C$ is perfect and meets all largest cliques of $G$.
\qed
\begin{theorem}\label{thm:PD-banner}
(Banner, odd hole)-free graphs are perfectly divisible.
\end{theorem}
{\it Proof}. Let $G$ be a (banner, odd hole)-free graph. We only
need prove $G$ is not minimally non-perfectly divisible. Suppose
$G$ is minimally non-perfectly divisible. We may assume $G$ is not
perfect, for otherwise we are done.  By
Observations~\ref{obs:PD-alpha}, we may assume $G$ contains a
co-triangle. By Theorem~\ref{thm:main} and the Strong Perfect Graph Theorem, $G$ contains a homogeneous
set. By Observation~\ref{obs:MNPD-homogeneous}, $G$ cannot be
minimally non-perfectly divisible, a contradiction. \qed

Theorem~\ref{thm:PD-banner} and Observation~\ref{obs:PD-chi-bounded} implies the following
corollary
\begin{corollary}
If $G$ is (banner, odd hole)-free graph, then $\chi(G) \leq \omega(G)^2$. \qed
\end{corollary}

Chudnovsky, Robertson, Seymour, and Thomas \cite{ChuRob2010}
proved that ($K_4$, odd hole)-free graphs are 4-colorable. It is
easily seen from this result that ($K_4$, odd hole)-free graphs
are perfectly divisible. (If $G$ is ($K_4$, odd hole)-free but imperfect, 
then $\chi(G) \leq 4 $ and $\omega(G) = 3$. Let $S_1, S_2, S_3, S_4$ be the color classes
of some 4-coloring of $G$. Then $S_1 \cup S_2$ is a compact set).  
Theorem~\ref{thm:PD-banner}  and the
result of Chudnovsky et al. suggest that perfectly divisible
graphs should be studied further. We note that recognizing perfectly divisible graphs is NP-complete.
\begin{lemma}\label{lem:np-hard}
	It is NP-complete to recognize perfectly divisible graphs. 
\end{lemma}
{\it Proof}. 
It is proved by Maffray and Priessmann \cite{MafPri1996} that it is NP-complete to recognize 3-colorable triangle-free graphs. We will reduce recognizing 3-colorable triangle-free graphs to recognizing perfectly divisible graphs. 
Let $G$ be a triangle-free graph. We may assume $G$ has an edge. We will show $G$ is 3-colorable if and only if $G$ is perfectly divisible. The ``only if'' part is trivial. For the ``if'' part, suppose $G$ is perfectly divisible and thus contains a compact set $C$. Since $G[G]$ is perfect and triangle-free, it is bipartite. Since $\omega(G) = 2$,   $G-C$ is a stable set. Thus, $G$ is 3-colorable.
 \qed

\section{Conclusions and open problems}
In this paper, we study the structure of (banner, odd hole)-free graphs. Our structure results show the existence of polynomial-time algorithms for recognizing, and finding a minimum coloring and a largest stable set of, (banner, odd hole)-free graphs. Actually, our Theorem~\ref{thm:main} implies the following: if WCOL for graphs with $\alpha = 2$ can be solved in polynomial time, then so can WCOL for   (banner, odd hole)-free graphs. We would like to propose this as an open problem.
\begin{problem}
	Determine the complexity of finding a minimum weighted coloring for graphs $G$ with $\alpha(G) = 2$.
\end{problem}
We have shown that (banner, odd hole)-free graphs are perfectly divisible. This result shows $\chi(G) \leq \omega(G)^2$ for a (banner, odd hole)-free graph $G$. We do not know of an odd-hole-free graph $G$ with $\chi(G) = \Omega(\omega(G)^2)$. 
We conclude our paper with the following problem. 
\begin{problem}
	Find an odd-hole-free graph that is not perfectly divisible.
\end{problem}
\begin{center}
	{\bf Acknowledgement}
\end{center}
The author thanks two anonymous referees for suggestions that improve the paper. 
This research is supported by an NSERC Discovery Grant.

\newpage
\begin{center}
	{\bf Appendix}
\end{center}
\begin{algorithm}
	\caption{STABLE-SET-UNION(G)}\label{alg:union}
	\begin{algorithmic}
		\STATE{\textbf{input:}} a disconnected graph $G$ with a weight function $w$ on its vertices.
		\STATE{\textbf{output:}} a maximum weighted stable set of $G$  or a message that $G$ is not (banner, odd hole)-free 
		\STATE
		
		\STATE compute the components $C_1, C_2, \ldots C_t$ of $G$
		\FOR{each component $C_i$, $i=1,2,\ldots,t$}
		\STATE call  WSS($G[C_i]$)
		\IF{a maximum weighted stable set $S_i$ of $G[C_i]$ is {\bf not} returned}
		\STATE return ``$G$ is not (banner, odd hole)-free'' and stop
		\ENDIF
		\ENDFOR
		\STATE let $S = S_1 \cup \ldots \cup S_t$ where $S_i$ is the stable set returned by WWS($G[C_i]$)
		\STATE return $S$ and stop
	\end{algorithmic}
\end{algorithm}

\begin{algorithm}
	\caption{STABLE-SET-JOIN(G)}\label{alg:join}
	\begin{algorithmic}
		\STATE{\textbf{input:}} a graph $G$ such that $\overline{G}$ is disconnected, with a weight function $w$ on its vertices.
		\STATE{\textbf{output:}} a maximum weighted stable set of $G$  or a message that $G$ is not (banner, odd hole)-free 
		\STATE
		\STATE compute the components $C_1, C_2, \ldots C_t$ of $\overline{G}$
		\STATE let $S=\emptyset$, $w(S) = 0$ 
		\FOR{each component $C_i$, $i=1,2,\ldots,t$}
		\STATE call  WSS($G[V(C_i)]$)
		\IF{a maximum weighted stable set $S_i$ of $G[V(C_i)]$ is not returned}
		\STATE return ``$G$ is not (banner, odd hole)-free'' and stop
		\ENDIF
		\IF {$w(S_i) > w(S)$ where $S_i$ is the stable set returned by WWS($G[C_i]$)  }
		\STATE set $S= S_i$, $w(S) = w(S_i)$
		\ENDIF
		\ENDFOR
		\STATE return $S$ and stop
		
	\end{algorithmic}
\end{algorithm}

\end{document}